\documentclass[a4paper,12pt]{article}
\usepackage{amsmath,amssymb,amsfonts,amsthm,enumerate,thmtools,hyperref,etoolbox}

\newtheorem{theorem}{Theorem}
\AfterEndEnvironment{theorem}{\noindent\ignorespaces}

\newtheorem{lemma}[theorem]{Lemma}
\AfterEndEnvironment{lemma}{\noindent\ignorespaces}

\AfterEndEnvironment{equation}{\noindent\ignorespaces}
\AfterEndEnvironment{align}{\noindent\ignorespaces}

\theoremstyle{definition}

\declaretheorem[style=definition,name=Remark,sibling=theorem]{remark}

\newcommand\testname{Acknowledgements}

\newenvironment{koszi}{%
    \small
    \begin{center}%
        {\bfseries \testname\vspace{-.5em}\vspace{0pt}}%
    \end{center}%
    \quotation}
    {\endquotation}

\newcommand{\eps}{\varepsilon}
\renewcommand{\(}{\left(}
\renewcommand{\)}{\right)}
\renewcommand{\o}{\overline}
\renewcommand{\c}{\cdot}

\newcommand{\zz}{\mathcal{Z}}

\newcommand{\summ}{\sum\limits}

\usepackage[margin=1.5in]{geometry}

\title{On the bipartite graph packing problem}
\author{B\'alint V\'as\'arhelyi\thanks{Szegedi Tudom\'anyegyetem, Bolyai Int\'ezet. Aradi v\'ertan\'uk tere 1., Szeged, 6720, Hungary, \url{mesti@math.u-szeged.hu}, Supported by T\'AMOP-4.2.2.B-15/1/KONV-2015-0006}}

\providecommand{\keywords}[1]{\textbf{\textit{Keywords:}} #1}

\begin{document}
\maketitle
\begin{abstract}
The graph packing problem is a well-known area in graph theory. We consider a bipartite version and give almost tight conditions on the packability of two bipartite sequences.
\end{abstract}


\section{Notation}

We consider only simple graphs. Throughout the paper we use common graph theory notations: $d_G(v)$ (or briefly, if $G$ is understood from the context, $d(v)$) is the degree of $v$ in $G$, and $\Delta(G)$ is the maximal and $\delta(G)$ is the minimal degree of $G$, and $e(X,Y)$ is the number of edges between $X$ and $Y$ for $X\cap Y = \emptyset$. For any function $f$ on $V$ let $f(X)=\summ_{v\in X}f(v)$ for every $X\subseteq V$. $\pi(G)$ is the degree sequence of $G$.

\section{Introduction}

Let $G$ and $H$ be two graphs on $n$ vertices. We say that $G$ and $H$ \emph{pack} if and only if $K_n$ contains edge-disjoint copies of $G$ and $H$ as subgraphs.

The graph packing problem can be formulated as an embedding problem, too. $G$ and $H$ pack if and only if $H$ is a subgraph of $\overline{G}$ ($H\subseteq \o G$).

A classical result is the theorem of Sauer and Spencer.

\begin{theorem}[Sauer, Spencer \cite{sauer}]
Let $G_1$ and $G_2$ be graphs on $n$ vertices with maximum degrees $\Delta_1$ and $\Delta_2$, respectively. If $\Delta_1\Delta_2<\frac{n}{2}$, then $G_1$ and $G_2$ pack.
\label{0:1}
\end{theorem}
Many questions in graph theory can be formulated as special packing problems, see \cite{kierstead}.
The main topic of the paper is a type of these packing questions, which is called degree sequence packing to be defined in the next section.
Some results in this field are similar to that of Sauer and Spencer (\autoref{0:1}).

The structure of the paper is as it follows. First, we define the degree sequence packing problem, and survey some results. Next, we state and prove our main result and also show that it is tight. In particular, we improve a bound given by Diemunsch et al.~\cite{ferrara} Finally, we consider some corollaries of our main theorem.

\section{Degree sequence packing}

\subsection{Graphic sequence packing}

Let $\pi = (d_1,\ldots,d_n)$ be a graphic sequence, which means that there is a simple graph $G$ with vertices $\{v_1,\ldots,v_n\}$ such that $d(v_i) = d_i$. We say that $G$ \emph{represents} $\pi$.

Havel \cite{havel} and Hakimi \cite{hakimi} gave a characterization of graphic sequences.

\begin{theorem} [Hakimi \cite{hakimi}]
Let $\pi = \{a_1,\ldots,a_n\}$ be a sequence of integers such that $n-1\geq a_1\geq \cdots\geq a_n\geq 0$. Then $\pi$ is graphic if and only if  by deleting any term $a_i$ and subtracting 1 from the first $a_i$ terms the remaining list is also graphic.
\end{theorem}
Kleitman and Wang \cite{kleitman} extended this result to directed graphs.

Two graphic sequences $\pi_1$ and $\pi_2$ \textit{pack} if there are graphs $G_1$ and $G_2$ representing $\pi_1$ and $\pi_2$, respectively, such that $G_1$ and $G_2$ pack. Obviously, the order does not matter.

There is an alternative definition to the packability of two graphic sequences. $\pi_1$ and $\pi_2$ \textit{pack with a fixed order} if there are graphs $G=(V,E_1)$ and $H=(V,E_2)$ with $V(\{v_1,\ldots,v_n\})$ such that $d_G(v_i)=\pi_1(i)$ and $d_H(v_i)=\pi_2(i)$ for all $i=1,\ldots,n$.

A detailed study of degree sequence packing we refer to Chapter 3 of Seacrest's PhD Thesis \cite{seacrest}.

One of the first results in (unordered or fixed order) degree sequence packing is the Lov\'asz--Kundu Theorem \cite{lovasz2, kundu}.

\begin{theorem}[Kundu \cite{kundu}]
A graphic sequence $\pi = (d_1,\ldots,d_n)$ has a realization containing a $k$-regular subgraph if and only if $\pi - k = (d_1-k,\ldots,d_n-k)$ is graphic.
\end{theorem}
Though we use the first definition, we give a result for the latter. Let $\Delta_i = \Delta(\pi_i)$ the largest degree and $\delta_i = \delta(\pi_i)$ the smallest degree of $\pi_i$ for $i=1,2$.

Busch et al.~\cite{busch} gave a condition for the packability of two graphic sequences with a fixed order. By $\pi_1+\pi_2$ they mean the vector sum of (the ordered) $\pi_1$ and $\pi_2$. 

\begin{theorem}[Busch et al. \cite{busch}]
Let $\pi_1$ and $\pi_2$ be graphic sequences of length $n$ with $\Delta = \Delta(\pi_1 + \pi_2)$ and $\delta = \delta(\pi_1+\pi_2)$. If $\Delta\leq \sqrt{2\delta n} - (\delta-1)$, then $\pi_1$ and $\pi_2$ pack with a fixed oreder. When $\delta = 1$, strict inequality is required.
\end{theorem}
Diemunsch et al.~\cite{ferrara} showed a condition for (unordered) graphic sequences.

\begin{theorem}[Diemunsch et al. \cite{ferrara}]\label{0:2}
Let $\pi_1$ and $\pi_2$ be graphic sequences of length $n$ with $\Delta_2 \geq\Delta_1$ and $\delta_1\geq 1$.

If 
\begin{equation}
\left\{\begin{array}{rl}
(\Delta_2+1)(\Delta_1+\delta_1) \leq \delta_1 n + 1,&when~\Delta_2+2\geq \Delta_1+\delta_1,~and\\
\dfrac{(\Delta_2+1+\Delta_1+\delta_1)^2}{4}\leq \delta_1n+1,&when~\Delta_2+2< \Delta_1+\delta_1,
\end{array}\right.
\end{equation} then $\pi_1$ and $\pi_2$ pack.

\end{theorem}

\subsection{Bipartite packing}

We study the bipartite packing problem as it is formulated by Catlin \cite{catlin}, Hajnal and Szegedy \cite{hajnal} and was used by Hajnal for proving deep results in complexity theory of decision trees \cite{hajnal2}.

Let $G_1 = (A,B;E_1)$ and $G_2 = (S,T;E_2)$ bipartite graphs with $|A|=|S|=m$ and $|B|=|T|=n$. They pack in the bipartite sense (i.e. they have a \textit{bipartite packing}) if there are edge-disjoint copies of $G_1$ and $G_2$ in $K_{m,n}$. 

Let us define the \textit{bigraphic sequence packing problem}.
We say that a sequence $\pi = (a_1,\ldots,a_m,b_1,\ldots,b_n)$ is \emph{bigraphic}, if $\pi$ is the degree sequence of a bipartite graph $G$ with vertex class sizes $m$ and $n$, respectively \cite{west}.

Two bigraphic sequences $\pi_1$ and $\pi_2$ without a fixed order \textit{pack}, if there are edge-disjoint bipartite graphs $G_1$ and $G_2$ with degree sequences $\pi_1$ and $\pi_2$, respectively, such that $G_1$ and $G_2$ pack in the bipartite sense.

Similarly to general graphic sequences, we can also define the packing with a fixed order.

Diemunsch et al.~\cite{ferrara} show the following for bigraphic sequences:

\begin{theorem}[Diemunsch et al. \cite{ferrara}]
\label{0:3}
Let $\pi_1$ and $\pi_2$ be bigraphic sequences with classes of size $r$ and $s$. Let $\Delta_1\leq \Delta_2$ and $\delta_1\geq 1$. If 
\begin{equation}
\Delta_1\Delta_2\leq\frac{r+s}{4},
\end{equation}
 then $\pi_1$ and $\pi_2$ pack.
\end{theorem}
The following lemma, formulated by Gale \cite{gale} and Ryser \cite{ryser}, will be useful. We present the lemma in the form as discussed in Lov\'asz, Exercise 16 of Chapter 7 \cite{lovasz}.

\begin{lemma}[Lov\'asz \cite{lovasz}]
\label{4}
	Let $G$ be a bipartite graph and $\pi$ a bigraphic sequence on $(A, B)$.
\begin{equation}
	\pi(X) \leq  e_G(X,Y) + \pi(\overline{Y})~~~\forall X\subseteq A,~\forall Y\subseteq B,
\end{equation}
   then $\pi$ can be embedded into $G$ with a fixed order.
\end{lemma}
For more results in this field, we refer the reader to the monography on factor theory of Yu and Liu \cite{yu}.

\section{Main result}

\begin{theorem}
\label{1}
For every $\eps\in(0,\frac{1}{2})$ there is an $n_0=n_0(\eps)$ such that if $n>n_0$, and $G(A,B)$ and $H(S,T)$ are bipartite graphs with $|A|=|B|=|S|=|T|=n$ and the following conditions hold, then $H\subseteq G$.

\begin{enumerate}
\setlength\itemindent{15ex}
	\item[\emph{Condition 1:}] \label{1:1} $d_G(x) > \(\frac{1}{2}+\eps\)n$ holds for all $x\in A\cup B$
	\item[\emph{Condition 2:}] $d_H(x) < \frac{\eps^4}{100}\frac{n}{\log n}$ holds for all $x\in S$,
	\item[\emph{Condition 3:}] $d_H(y)=1$ holds for all $y\in T$.
\end{enumerate}
\end{theorem}
We prove \autoref{1} in the next section. First we indicate why we have the bounds in Conditions 1 and 2.

Condition 1 of \autoref{1} is necessary. Suppose that $\frac{n}{2}-1 < d_G(x)$. That allows $G = K_{\frac{n}{2}+1,\frac{n}{2}-1}\cup K_{\frac{n}{2}-1,\frac{n}{2}+1}$. For all $\eps>0$ there is an $n_0$ such that if $n>n_0$ degrees are higher than $\left(\frac{1}{2}-\eps\right)n$, but there is no perfect matching (i.e. 1-factor) in the graph.

Condition 2 is necessary as well. To show it, we give an example. Let $G = G(n,n,p)$ a random bipartite graph with $p>0.5$ and vertex class sizes of $n$. Let $H(S,T)$ be the following bipartite graph: each vertex in $T$ has degree $1$. In $S$ all vertices have degree 0, except $\frac{\log n}{c}$ vertices with degree $\frac{cn}{\log n}$. The graph $H$ cannot be embedded into $G$, which follows from the example of Koml\'os et al. \cite{komlos}

%
%
%
%
%
%
%
%
%

Before proving \autoref{1} we compare our main theorem with the previous results.

\begin{remark}
There are graphs which can be packed using \autoref{1}, but not with \autoref{0:1}.
\end{remark}

Indeed, $\Delta_1 >\frac{n}{2} $ and we can choose $\Delta_2 >1$. Thus, $\Delta_1\Delta_2>\frac{n}{2}$. However, with \autoref{1} we can pack $G$ and $H$.

\begin{remark}
There are graphs which can be packed using \autoref{1}, but not with \autoref{0:2}.
\end{remark}

Let $\pi_1 = \pi(H)$ and $\pi_2 = \pi(\overline{G})$.

$\delta_1 = 1$ and $\Delta_1\leq \frac{n}{100\log n}$.

If $\Delta_2 \approx \frac{n}{2}$, then $\Delta_2+2 \geq \Delta_1+\delta_1$.

Furthermore, \begin{equation}
(\Delta_2+1)(\Delta_1+\delta_1) \approx \frac{n}{2}\cdot \frac{n}{c\log n} \gg n.
\end{equation}
Although the conditions of \autoref{0:2} are not satisfied, $\pi_1$ and $\pi_2$ still pack.

\begin{remark}
There are graphs which can be packed using \autoref{1}, but not with \autoref{0:3}.
\end{remark}

Let $\pi_1 = \pi(H)$ and $\pi_2 = \pi(\o G)$, as above. The conditions of \autoref{0:3} are not satisfied, however, \autoref{1} gives a packing of them.

As it is transparent, our main theorem can guarantee packings in cases, that were far beyond reach by the previous tecniques.

\section{Proof}

We formulate the key technical result for the proof of \autoref{1} in the following lemma.

\begin{lemma}
		\label{5}
		Let $\eps$ and $c$ such that in \autoref{1}. Let $G$ and $H$ be bipartite graphs with classes $Z$ and $W$ of sizes $z$ and $n$, respectively, where $z>\frac{2}{\eps}$.
		
		Suppose that
		\begin{enumerate}
			\item $d_G(x)>\(\frac{1}{2}+\eps\)n$ for all $x\in Z$ and
			\item $d_G(y)>\(\frac{1}{2}+\frac{\eps}{2}\)z$ for all $y\in W$.
		    \newcounter{enumTemp}
    		\setcounter{enumTemp}{\theenumi}
		\end{enumerate}
		Assuming
		\begin{enumerate}
		    \setcounter{enumi}{\theenumTemp}
			\item There is an $M\in \mathbb{N}$ and with $\delta\leq \frac{\eps}{10}$ we have 
			$$M\leq d_H(x)\leq M(1+\delta)~\forall x\in Z,$$
			and
			\item $d_H(y) = 1~\forall y\in W$.
		\end{enumerate}
		Then there is an embedding of $H$ into $G$.
\end{lemma}

\begin{proof}
	We show that the conditions of \autoref{4} are satisfied.

	Let $X\subseteq Z$, $Y\subseteq W$. We have five cases to consider depending on the size of $X$ and $Y$.
	
	In all cases we will use the obvious inequality $Mz\leq n$, as $d_H(X) = d_H(Y)$. For sake of simplicity, we use $e(X,Y)=e_G(X,Y)$.
\begin{enumerate}[(a)]
\item
	$|X|\leq\frac{z}{2(1+\delta)}$ and $|Y|\leq\frac{n}{2}$.
	
	We have
	
	\begin{equation}
	d_H(X)\leq M(1+\delta)|X| \leq M(1+\delta)\frac{z}{2(1+\delta)} = \frac{Mz}{2}\leq \frac{n}{2} \leq |\overline{Y}|=d_H\(\overline{Y}\).
	\end{equation}
	
\item $|X|\leq\frac{z}{2(1+\delta)}$ and $|Y|>\frac{n}{2}$.

	Let $\phi = \frac{|Y|}{n}-\frac{1}{2}$, so $|Y| = \(\frac{1}{2}+\phi\)n$. Obviously, $0\leq\phi\leq\frac{1}{2}$.
	
	Therefore, $d_H(\o Y)=|\o Y| = \(\frac{1}{2}-\phi\)n$.
	
	Since $d_H(X)\leq \frac{n}{2}$, as we have seen above, furthermore,
	
	\begin{equation}
	e(X,Y)\geq (\eps + \phi)n|X|\geq (\eps+ \phi)n,
	\end{equation}
	we obtain $d_H(X)\leq d_H(\o Y)+e_G(X,Y)$.
	
\item $\frac{z}{2}\geq|X|>\frac{z}{2(1+\delta)}$ and $|Y|\leq\frac{n}{2}$.
	\label{case_c}

	Let $\psi = \frac{|X|}{z}-\frac{1}{2(1+\delta)}$, hence, $|X| = \(\frac{1}{2(1+\delta)}+\psi\)z$. Let $\psi_0 = \frac{\delta}{2(1+\delta)}=\frac{1}{2}-\frac{1}{2(1+\delta)}$, so $\psi\leq \psi_0$. This means that $|X|=\(\frac{1}{2}-\psi_0+\psi\)z$.
	
	As $0<\delta\leq\frac{\eps}{10}$, we have $\psi_0<\frac{\delta}{2}\leq \frac{\eps}{20}$.
	
	Let $\phi = \frac{1}{2}-\frac{|Y|}{n}$, so $|Y| = \(\frac{1}{2}-\phi\)n$. As $|Y|\leq\frac{n}{2}$, this gives $0\leq \phi\leq \frac{1}{2}$.
	
	\begin{enumerate}[(1)]
	\item $d_H(\o Y)=|\o Y| = n\(\frac{1}{2}+\phi\)$
	\item As above, $d_H(X)\leq M(1+\delta)|X| = Mz(1+\delta)\(\frac{1}{2(1+\delta)}+\psi\)\leq n(1+\delta)\(\frac{1}{2(1+\delta)}+\psi\)$.
	\item We claim that $e(X,Y)\geq |Y|\(\frac{\eps}{2}-\psi_0+\psi\)z$. Indeed, the number of neighbours of a vertex $y\in Y$ in $X$ is at least $\(\frac{\eps}{2}+\psi-\psi_0\)z$, considering the degree bounds of $W$ in $H$.
	\end{enumerate}

	We show $d_H(X)\leq e(X,Y) + d_H(\o Y)$.
	
	It follows from
	
	\begin{equation}
	n(1+\delta)\(\frac{1}{2(1+\delta)} + \psi\) \leq n\(\frac{1}{2}-\phi\)\(\frac{\eps}{2}-\psi_0+\psi\)z + n\(\frac{1}{2}+\phi\).
	\end{equation}
	
	This is equivalent to
		
	\begin{equation}
	\psi + \delta\psi\leq z\(\frac{1}{2}-\phi\)\(\frac{\eps}{2}+\psi-\psi_0\) + \phi.
	\label{eq1}
	\end{equation}
	
	The left hand side of \eqref{eq1} is at most $\psi_0 + \delta\psi_0 \leq \frac{\delta}{2}+\frac{\delta^2}{2}\leq \delta$, as $\delta\leq\eps\leq\frac{1}{2}$.
	
	If $\phi > \delta$, \eqref{eq1} holds, since $\frac{\eps}{2}+\psi-\psi_0\geq 0$, using $\psi_0\leq\frac{\eps}{20}$.
	
	Otherwise, if $\phi\leq \delta$, the right hand side of \eqref{eq1} is
	
	\begin{equation}
	z\(\frac{1}{2}-\phi\)\(\frac{\eps}{2}+\psi-\psi_0\) \geq \(\frac{1}{2}-\delta\)\(\frac{\eps}{2}-\frac
	{\delta}{2}\)z.
	\end{equation}
	We also have
	\begin{equation}
	\frac{\eps}{4}+\frac{\delta^2}{2}-\frac{\delta\eps}{2}-\frac{\delta}{4} > \delta,
	\end{equation}
	since
	\begin{equation}
	\eps + 2\delta^2-2\delta\eps > \eps - 2\frac{\eps^2}{10}>\frac{\eps}{2}>5\delta,
	\end{equation}
	using $\delta\leq \frac{\eps}{10}$.
	
	This completes the proof of this case.

\item $|X|>\frac{z}{2}$ and $|Y|\leq\frac{n}{2}$.

		We have
		
		\begin{enumerate}[(1)]
		\item $d_H(X)= d_H(Z)-d_H(\o X) = n - d_H(\o X) \leq n - M|\o X|,$
		\item $d_H(\o Y) = n- |Y|$ and
		\item $e(X,Y) \geq |Y|\(|X|-\frac{z}{2}+\frac{\eps z}{2}\)$, using to the degree bound on $Y$.
		\end{enumerate}
		
		All we have to check is whether
		
		\begin{equation}
		n - M|\o X| \leq n - |Y| + |Y|\(|X|-\frac{z}{2} + \frac{\eps z}{2}\)
		\end{equation}
		It is equivalent to
		\begin{equation}
		0 \leq |Y|\(|X|-\frac{z}{2}+\frac{\eps z}{2}-1\)+M\(z-|X|\)
		\label{eq4}
		\end{equation}
		\eqref{eq4} has to be true for any $Y$ and $M$. Specially, with $|Y|=M=1$, \eqref{eq4} has the following form:
		
		\begin{equation}
		0\leq |X| - \frac{z}{2} + \frac{\eps z}{2}-1 + z - |X| = \frac{z}{2}+\frac{\eps z}{2}-1.
		\label{eq5}
		\end{equation}
		\eqref{eq5} is true if $z\geq 2$.
		
		If $z = 1$, then $Z=\{v\}$ is only one vertex, which is connected to each vertex in $W$. In this case, \autoref{5} is obviously true.
		
\item $|X|>\frac{z}{2(1+\delta)}$ and $|Y|>\frac{n}{2}$.
	
	Let $\psi = \frac{|X|}{z}-\frac{1}{2(1+\delta)}$, hence, $|X| =z\(\frac{1}{2(1+\delta)}+\psi\)$. Let $\psi_0 = \frac{\delta}{2(1+\delta)}$, as it was defined in Case \eqref{case_c}. Again, $\psi_0\leq\frac{\delta}{2}$. We have $0\leq\psi\leq\frac{1}{2}+\psi_0\leq \frac{1+\delta}{2}$.
	
	Let $\phi = \frac{|Y|}{n}-\frac{1}{2}$, hence, $|Y| = n\(\frac{1}{2}+\phi\)$.
	 
	We have
	
	\begin{enumerate}[(1)]
	\item $d_H(X)\leq zM(1+\delta)\(\frac{1}{2(1+\delta)}+\psi\)\leq n (1+\delta)\(\frac{1}{2(1+\delta)}+\psi\),$
	\item $d_H(\o Y) = n\(\frac{1}{2}-\phi\)$ and
	\item $e(X,Y)\geq z\(\frac{1}{2(1+\delta)}+\psi\)(\phi+\eps)n.$
	\end{enumerate} 
	From the above it is sufficient to show that
	
	\begin{equation}
	n(1+\delta)\(\frac{1}{2(1+\delta)}+\psi\)\leq n\(\frac{1}{2}-\phi\) + z\(\frac{1}{2(1+\delta)+\psi}\)(\phi+\eps)n.
	\end{equation}
	It is equivalent to
	
	\begin{equation}
	\psi(1+\delta)\leq -\phi + z\(\frac{1}{2(1+\delta)}+\psi\)(\phi + \eps).
	\label{eq2}
	\end{equation}
	
	Using $\psi\leq\frac{1+\delta}{2}$ and $\delta\leq\frac{\eps}{10}$, the left hand side of \eqref{eq2} is at most
	
	\begin{equation}
	\frac{1+\delta}{2}(1+\delta) = \frac{1}{2}+\delta+\frac{\delta^2}{2}\leq \frac{1}{2}+\frac{\eps}{10}+\frac{\eps^2}{200}\leq\frac{1}{2}+\frac{1}{10}=\frac{3}{5},
	\end{equation}
	as $\eps\leq\frac{1}{2}$.
	
	The right hand side of \eqref{eq2} is
	
	\begin{equation}
	\phi\frac{z-2(1+\delta)}{2(1+\delta)} + \frac{z}{2(1+\delta)}\eps + z\psi(\phi+\eps)
	\label{eq3}
	\end{equation}
	
	The first and the last term of \eqref{eq3} is always positive. (We use that $z>3$.) Therefore, \eqref{eq3} is at least $\frac{z}{2(1+\delta)}\eps$.
	
	It is enough to show that
	
	\begin{equation}
	\frac{3}{5} \leq \frac{z}{2(1+\delta)}\eps.
	\end{equation}
	This is true indeed, since $\eps>\frac{2}{z}$ and $\delta\leq\frac{\eps}{10}\leq\frac{1}{20}$.
	
	We have proved what was desired.

\end{enumerate}
	
\end{proof}

\begin{proof} (\autoref{1})
First, form a partition $C_0,C_1,\ldots,C_k$ of $S$ in the graph $H$. For $i>0$ let $u\in C_i$ if and only if $\frac{\eps^4}{100}\frac{n}{\log n}\c\frac{1}{(1+\delta)^{i-1}}\geq d_H(u)>\frac{\eps^4}{100}\frac{n}{\log n}\c\frac{1}{(1+\delta)^{i}}$ with $\delta = \frac{\eps}{10}$. Let $C_0$ be the class of the isolated points in $S$. Note that the number of partition classes, $k$ is $\log_{1+\delta}n = \log_{1 + \frac{\eps}{10}} n =\frac{\log n}{\log \(1 + \frac{\eps}{10}\)} = c\log n$.


Now, we embed the partition of $S$ into $A$. Take a random ordering of the vertices in $A$. The first $|C_1|$ vertices of $A$ form $A_1$, the vertices $|C_1|+1,\ldots,|C_1|+|C_2|$ form $A_2$ etc., while $C_0$ maps to the last $|C_0|$ vertices. Obviously, $C_0$ can be always embedded.

We say that a partition class $C_i$ is \emph{small} if $|C_i|\leq \frac{16}{\eps^2}\log n$.

We claim that the total size of the neighbourhood in $B$ of small classes is at most $\frac{\eps n}{4}$. 

The size of the neighbourhood of $C_i$ is at most \begin{equation}
\frac{\eps^4}{100}\frac{n}{\log n}\c \frac{1}{(1+\delta)^{i-1}}\c\frac{16}{\eps^2}\log n.
\end{equation}
If we sum up, we have that the total size of the neighbourhood of small classes is at most

\begin{align}
\sum_{i=1}^{k}\frac{\eps^4}{100}\frac{n}{\log n}\c \frac{1}{(1+\delta)^{i-1}}\c\frac{16}{\eps^2}\log n = \frac{4}{25}\eps^2n\sum_{i=0}^{k-1}\frac{1}{(1+\delta)^i}\leq \nonumber\\
\leq\frac{4}{25}\eps^2n\frac{1+\delta}{\delta}\leq\frac{4}{25}\eps^2n\frac{3/20}{\eps/10}\leq \frac{\eps n}{4}.\phantom{mmmmmmmmmmm}
\end{align}
The vertices of the small classes can be dealt with using a greedy method: if $v_i$ is in a small class, choose randomly $d_H(v_i)$ of its neighbours, and fix these edges. After we are ready with them, the degrees of the vertices of $B$ are still larger than $\(\frac{1}{2}+\frac{\eps}{2}\)n$.

Continue with the large classes. Reindex the large classes $D_1,\ldots,D_\ell$ and form a random partition $E_1,\ldots,E_\ell$ of the unused vertices in $B$ such that $|E_i| = \summ_{u\in D_i} d_H(u)$. We will consider the pairs $(D_i,E_i)$.

We will show that the conditions of \autoref{5} are satisfied for $(D_i,E_i)$.

For this, we will use the Azuma--Hoeffding inequality.

We have to show that for any $i$ every vertex $y\in E_i$ has at least $\(\frac{1}{2}+\frac{\eps}{4}\)z$
neighbours in $D_i$ and every vertex $x\in D_i$ has at least $\(\frac{1}{2}+\frac{\eps}{2}\)z$ in $E_i$.

Then we apply \autoref{5} with $\frac{\eps}{2}$ instead of $\eps$, and we have an embedding in each pair $(D_i,E_i)$, which gives an embedding of $H$ into $G$.

Let $|D_i| = z$. We know $z > \frac{16}{\eps^2} \log n$, as $D_i$ is large.

Build a martingale $\zz = \zz_0,\zz_1,\ldots,\zz_z$. Consider a random ordering $v_1,\ldots,v_z$ of the vertices in $Z$. Let $X_i=1$ if $v_i$ is a neighbour of $y$, otherwise, let $X_i = 0$. Let $\zz_i = \summ_{j=1}^i X_j$, and let $\zz_0 = 0$. This chain $\zz_i$ is a martingale indeed with martingale differences $X_i\leq 1$, which is not hard to verify.

According to the Azuma--Hoeffding inequality \cite{azuma, hoeffding} we have the following lemma:

\begin{lemma}[Azuma \cite{azuma}]
If $\zz$ is a martingale with martingale differences $1$, then for any $j$ and $t$ the following holds:
\label{azulemma}
\begin{equation}
\mathbb{P}\(\zz_j\geq\mathbb{E}\zz_j-t\)\geq 1 - e^{-\frac{t^2}{2j}}.
\end{equation}
\end{lemma}
The conditional expected value $\mathbb{E}(\zz_z|\zz_0)$ is $\mathbb{E}\zz_z = \(\frac{1}{2}+\frac{3\eps}{4}\)z$.

\autoref{azulemma} shows that

\begin{equation}
\mathbb{P}\(\zz_z\geq\(\frac{1}{2}+\frac{\eps}{2}\)z\) \geq 1 - e^{-\frac{\eps^2z^2/4}{2z}} = 1-e^{-\eps^2 z / 8}.
\label{azueq}
\end{equation}
We say that a vertex $v\in E_i$ is \textbf{bad}, if it has less than $\(\frac{1}{2}+\frac{\eps}{2}\)z$ neighbours in $D_i$. \autoref{azulemma} means that a vertex $v$ is bad with probability at most $e^{-\eps^2 z/8}$. As we have $n$ vertices in $B$, the probability of the event that any vertex is $C$-bad is less than 
\begin{equation}
n\cdot e^{-\eps^2z/8}<\frac{1}{n},
\end{equation}
as $z>\frac{16}{\eps^2}\log n$.

Then we have that with probability $1-\frac{1}{n}$ no vertex in $E_i$ is bad. Thus, Condition (ii) of \autoref{5} is satisfied with probability 1 for any pair $(D_i,E_i)$.

Using \autoref{azulemma}, we can also show that each $x\in D_i$ has at least $\(\frac{1}{2}+\frac{\eps}{2}\)|E_i|$ neighbours in $E_i$ with probability 1.

Thus, the conditions of \autoref{5} are satisfied, and we can embed $H$ into $G$. The proof of \autoref{1} is finished.
\end{proof}

%
%
%
%
%
\clearpage
\begin{koszi}
I would like to thank my supervisor, B\'ela Csaba his patient help, without whom this paper would not have been written. I also express my gratitude to P\'eter L. Erd\H os and to P\'eter Hajnal thoroughly for reviewing and correcting the paper. This work was supported by T\'AMOP-4.2.2.B-15/1/KONV-2015-0006
\end{koszi}
\clearpage
\bibliography{On_the_bipartite_degree_sequence_packing_problem}
\bibliographystyle{amsplain}
\end{document}